\documentclass{article}[12pt]
\usepackage{amsmath}
\usepackage{amssymb}
\usepackage{amsfonts}
\usepackage{mathtools}
\usepackage[top=2cm,bottom=2cm,left=2.5cm,right=2cm]{geometry}

\usepackage{amsmath}
\usepackage{graphicx}
\usepackage[colorinlistoftodos]{todonotes}
\usepackage[colorlinks=true, allcolors=blue]{hyperref}
\usepackage{amsthm}
\usepackage{bbm}
\usepackage{enumerate}
\usepackage{verbatim}
\usepackage{biblatex}
\usepackage{float}
\usepackage{tikz-cd}
\usepackage[OT2,T1]{fontenc}
\usepackage{xcolor}
\usepackage{tcolorbox}
\usepackage{adjustbox}
\usepackage{booktabs}

\newtcolorbox{commentbox}[1][]{colback=yellow!10!white, colframe=red!50!black, title=#1}
\DeclareSymbolFont{cyrletters}{OT2}{wncyr}{m}{n}
\DeclareMathSymbol{\Sha}{\mathalpha}{cyrletters}{"58}

\newtheorem{theorem}{Theorem}[section]

\newtheorem{fact}[theorem]{Fact}
\newtheorem{lemma}[theorem]{Lemma}
\newtheorem{proposition}[theorem]{Proposition}
\newtheorem{corollary}[theorem]{Corollary}
\newtheorem{definition}[theorem]{Definition}

\newtheorem{remark}[theorem]{Remark}
\newtheorem{example}[theorem]{Example}
\newtheorem{notation}[theorem]{Notation}

\newtheorem{question}[theorem]{Question}
\newtheorem{algorithm}[theorem]{Algorithm}

\bibliography{main}
\author{Sa'ar Zehavi}
\title{Modular Chabauty: Effective $S$-Integral Point Computation On Curves with Elliptic Fibrations}
\begin{document}

\maketitle

\begin{abstract}
We present a practical, unconditional algorithm for determining the $S$-integral points on any \emph{elliptic moduli problem} $\mathcal Y/\mathbb Z[1/S]$---that is, on any geometrically connected curve carrying a non-isotrivial elliptic fibration $\pi:\mathcal E\to\mathcal Y$. The associated map $\Phi_{\mathcal M}:\mathcal Y\to\mathcal M_{1,1}$ (the \emph{modular period map}) plays the role ordinarily filled by a $p$-adic period map in Chabauty--type methods. 

Our \emph{Modular Chabauty} method studies the image and fibres of $\Phi_{\mathcal M}$, and proceeds in two steps: an Effective Shafarevich step, in which we combine the modularity theorem with Cremona’s enumeration of elliptic curves by conductor, and list \emph{all} rational elliptic curves with good reduction outside~$S$; and a Fibre Computation step, in which we compute the $S$-integral points in the corresponding fibre of~$\Phi_{\mathcal M}$.

A Python/Sage implementation computes $\mathcal Y(\mathbb Z[1/S])$ for \(\mathcal Y=\mathbb{P}^1\!\setminus\!\{0,1,\infty\}\) and for every modular curve $\mathcal Y_1(N)$ with \(4\le N\le10\) or \(N=12\), for all sets \(S\) with \(\prod_{p\in S}p^{2}\le5\!\cdot\!10^{5}\), within $3.5$ seconds on a standard computer.
\end{abstract}

\tableofcontents

\section{Introduction}
\label{chapter:intro}
Let $K$ be a number field and $S$ a finite set of places of $K$ containing all archimedean places. Denote by $\mathcal{O}_{K,S}$ the ring of $S$-integers in $K$. A classic corollary of Siegel~\cite{SIEGEL} and Faltings~\cite{FALTINGS} famous theorems reads: if a curve is \emph{hyperbolic}---i.e.\ it has \emph{negative Euler characteristic}---then its $S$-integral points form a \emph{finite} set. The open challenge is \textbf{effectivity}: 
\begin{question}
Given a hyperbolic curve
\(
\mathcal{Y}/\mathcal{O}_{K,S},
\)
\emph{compute} the set
\(
\mathcal{Y}(\mathcal{O}_{K,S}).
\)  
\end{question}
Here, "compute" means providing an algorithm that computes $\mathcal{Y}(\mathcal{O}_{K,S})$ from the data $(K, S, \mathcal{Y})$ in finite time, and preferably efficiently.

For an integer $g > 0$, let $\mathcal{M}_{g}$ denote the moduli stack of principally polarized $g$-dimensional abelian varieties. Its $\mathcal{O}_{K,S}$-points classify abelian varieties of dimension $g$ defined over $K$ with good reduction outside $S$. The Shafarevich Conjecture, now a theorem of Faltings, asserts that for any fixed $K$, $S$, and $g$, the set $\mathcal{M}_{g}(\mathcal{O}_{K,S})$ is finite. Parshin~\cite{PARSHIN} first recognized a connection between the Shafarevich and Mordell conjectures, subsequently exploited by Faltings~\cite{FALTINGS} in proving Mordell's conjecture. A simplified outline of the idea is as follows: Given a $g'$-dimensional relative abelian variety $\pi: \mathcal{X}\to \mathcal{Y}$ over a smooth geometrically connected $\mathcal{O}_{K,S}$-scheme $\mathcal{Y}$, we obtain a map:
\[
\Phi(\mathcal{O}_{K,S}): \mathcal{Y}(\mathcal{O}_{K,S}) \to \mathcal{M}_{g'}(\mathcal{O}_{K,S}), \quad t\mapsto \pi^{-1}(t).
\]
From the Shafarevich Conjecture, it follows immediately that $\Phi(\mathcal{O}_{K,S})$ has finite image. Under favorable conditions, such as when $\pi$ is a Kodaira--Parshin fibration, Faltings proved that $\Phi$ has finite fibres, hence $\mathcal{Y}(\mathcal{O}_{K,S})$ is finite.

The \emph{Effective Shafarevich Problem} then asks:
\begin{question}
Given $g, K$, and $S$ as above, explicitly compute $\mathcal{M}_{g}(\mathcal{O}_{K,S})$.
\end{question}
Effective solutions to the Shafarevich Problem lead to effective solutions of the Faltings--Siegel Problem through Kodaira--Parshin families, as demonstrated by~\cite{JAVANPEYKAR, ALPOGE, KANEL_MATSCHKE, VON_KANEL}. Consider a smooth hyperbolic curve $\mathcal{Y}$ equipped with a Kodaira--Parshin fibration $\mathcal{X}\to \mathcal{Y}$ defined over $\mathcal{O}_{K,S}$. Assume for simplicity that $\mathcal{X}\to \mathcal{Y}$ is a relative abelian variety of dimension $g'>0$. A canonical example occurs when $\mathcal{Y}$ is a modular curve and $\mathcal{X}$ is the universal elliptic curve. Denote the induced morphism to the moduli stack $\mathcal{M}_{g'}$ by $\Phi_{\mathcal{M}}: \mathcal{Y}\to \mathcal{M}_{g'}$, which we call the \emph{Modular Period Map}.

\begin{definition}
An $S$-integral elliptic moduli problem is a pair $(\mathcal{Y}, \Phi_{\mathcal{M}_{1,1}})$, consisting of a geometrically connected curve $\mathcal{Y}/\mathcal{O}_{K,S}$, and a non-constant map $\Phi_{\mathcal{M}_{1,1}}: \mathcal{Y}\to \mathcal{M}_{1,1}$, to the moduli stack of elliptic curves, or, equivalently, a non-isotrivial elliptic fibration $\mathcal{E}\to \mathcal{Y}$. 
\end{definition}

\begin{remark}
Elliptic moduli problems only ever exist over affine curves, since if $\mathcal{E}\to \mathcal{Y}$ is a non-isotrivial elliptic fibration of a projective curve, then composing $\Phi_{\mathcal{M}_{1,1}}: \mathcal{Y}\to \mathcal{M}_{1,1}$ with the $j$-map, $j: \mathcal{M}_{1,1}\to \mathbb{A}^1$ gives a non-constant global function on $\mathcal{Y}$, a contradiction.
\end{remark}

\begin{remark}
In the case of an elliptic moduli problem, we shall refer to $\Phi_{\mathcal{M}_{1,1}}$ as the Modular Period Map, and denote it by $\Phi_{\mathcal{M}}$.
\end{remark}

Analogous to the Chabauty--Coleman, Chabauty--Kim, and Lawrence--Venkatesh methods, which study the image and fibres of certain $p$-adic period maps, \emph{Modular Chabauty} is a two-step strategy analyzing the image and fibres of the Modular Period Map:
\begin{enumerate}
    \item \textbf{Effective Shafarevich Step}: Compute $\mathcal{M}_{g'}(\mathcal{O}_{K,S})$ explicitly.
    \item \textbf{Fibre Computation Step}: For each $\mathcal{A}_{g'}\in \mathcal{M}_{g'}(\mathcal{O}_{K,S})$, explicitly determine $\Phi^{-1}_{\mathcal{M}}(\mathcal{A}_{g'})(\mathcal{O}_{K,S})$.
\end{enumerate}
In this paper, we implement this strategy for \emph{elliptic moduli problems} over open $\mathrm{Spec}\,\mathbb{Z}$-subschemes. In that setting the Effective Shafarevich step can be resolved by invoking the full \emph{Modularity Theorem} for elliptic curves over $\mathbb{Q}$~\cite{WILES, TAYLOR-WILES, BCDT} (see Theorem~\ref{theorem:modularity}), in combination with Cremona's algorithms~\cite{CREMONA} for modular elliptic curves. While discovered independently by the author, an essentially identical approach to the Effective Shafarevich step previously appeared in the work of von Känel and Matschke~\cite{KANEL_MATSCHKE} with application to various elliptic moduli problems. While our algorithm synthesizes known individual ideas, our originality lies in introducing an effective, conceptually simple, and universally applicable fibre computation step. We also introduce a moduli-theoretic approach in the fibre computation step for modular curves, providing a model-independent description of their $S$-integral points. This may be especially useful to those aiming to prove the (in)existence of such points.

Our main contribution is a practical and explicit algorithm (Algorithm~\ref{algorithm:mcemp}), efficiently computing $\mathcal{Y}(\mathbb{Z}[1/S])$ for elliptic moduli problems. In \S~\ref{chapter:tables}, we apply this method to the thrice-punctured line $\mathbb{P}^1\setminus\{0,1,\infty\}$ and modular curves with $\Gamma_1(N)$-level structure for $4\le N\le 10$ and $N=12$\footnote{Note that $\mathcal{Y}_1(N)(\mathbb{Q}) = \emptyset$ for $N$ outside this range as a corollary to Mazur's Torsion Theorem~\cite[Theorem 2]{MAZUR_ISOGENY}.}. Our Python-Sage implementation~\cite{MY_IMPLEMENTATION} computes these solutions within a few seconds on a standard personal computer, subject to $\prod_{p\in S}p^2 \lesssim 5\cdot 10^5$.

We view Modular Chabauty as part of a broader effort toward systematically addressing the Effective Faltings--Siegel Problem via modularity, supported by recent advances such as those by Alpöge~\cite{ALPOGE}.

\subsection{Related Work}
\label{section:related_work}
Proofs of Diophantine finiteness using Diophantine approximation methods, dating back to Siegel's original approach~\cite{SIEGEL}, typically rely on variants of Baker's method. Such techniques yield explicit height bounds for $S$-integral points for some affine hyperbolic curves (e.g. in the genus $0, 1$ cases), theoretically enabling exhaustive search. However, the numerical bounds derived from these methods are usually impractically large, rendering such computations infeasible. For an extensive survey on Diophantine approximation methods relevant to the Effective Siegel Problem, see Bilu~\cite{BILU}.

The idea of studying modular period maps and their fibres as a means to establish Diophantine finiteness traces back to Parshin~\cite{PARSHIN}. Parshin's vision was realized by Faltings~\cite{FALTINGS}, who reduced the Mordell conjecture to the Shafarevich conjecture and proved the latter (hence proving both). Although Faltings' original proof was non-effective, subsequent developments~\cite{JAVANPEYKAR, ALPOGE, KANEL_MATSCHKE, VON_KANEL} demonstrated that effective solutions to the Shafarevich Problem imply effective solutions to the Faltings--Siegel Problem. Von Känel~\cite{VON_KANEL} notably utilized modularity results to establish explicit height bounds on several families of abelian varieties over $\mathrm{Spec}\, \mathbb{Z}[1/S]$, from which he produced explicit but typically large height bounds on $S$-integral points for elliptic moduli problems. Despite improvement over earlier bounds (e.g., those given by Bilu~\cite{BILU}), these remain too large for practical computation (via exhaustive search).

Our Effective Shafarevich step, based on modularity results and Cremona's algorithms, though discovered independently by the author, previously appeared in the work of von Känel and Matschke. In~\cite{KANEL_MATSCHKE}, von Känel and Matschke introduced two practical algorithms for explicitly computing $S$-integral points on particular families of elliptic moduli problems, such as cubic Thue/Thue--Mahler, Mordell, Ramanujan--Nagell, and the thrice-punctured line. Their first algorithm can be viewed as a form of \emph{Modular Chabauty}, consisting of two steps: (i) an Effective Shafarevich step identical to ours, and (ii) a case-by-case analysis for the fibre computation step. In view of this, our method may be naturally regarded as building upon their modularity-based framework. The key difference between our methods lies in our Fibre Computation step: we replace their specialized, model-dependent analyses with a conceptually simple approach, valid for any elliptic moduli problem. In the special case of modular curves, we introduce a model-independent, moduli-theoretic description of $S$-integral points, which seems advantageous for theoretical investigations into the (in)existence of such points.

Von Känel and Matschke's second algorithm, based on explicit height bounds combined with a sieve method originally introduced by de Weger~\cite{de_WEGER}, currently holds the computational world record for solving the $S$-unit equation for the largest set of primes $S(16)$ (consisting of the first $16$ primes). According to~\cite[\S 3.2.6]{KANEL_MATSCHKE}, this record computation required 34 days on 30 CPUs. In comparison, our implementation currently handles sets $S$ subject to the constraint $\prod_{p\in S}p^2 \lesssim 5\cdot 10^5$. This limitation arises since (i) the elliptic curves identified in our Effective Shafarevich step have conductors roughly bounded by $\prod_{p\in S}p^2$, and (ii) Cremona's tables~\cite{CREMONA_TABLES}, which we rely on, store isomorphism classes of rational elliptic curves up to conductor $5\cdot 10^5$. Within this practical constraint, our Python-Sage implementation~\cite{MY_IMPLEMENTATION} efficiently solves the $S$-unit equation for sets of cardinality at most $3$; for example, finding all $39$ solutions for $S = {2,3,23}$ in approximately $2.3$ seconds on a personal computer. Extending our method to sets $S$ of cardinality $4$ would require Cremona tables roughly $100$ times larger, which currently do not exist. To handle significantly larger sets, such as $S(16)$ or beyond, one would need a dynamic implementation of the Effective Shafarevich step, generating Cremona's tables only for relevant conductors as needed. Exploring such an implementation and comparing its performance with von Känel and Matschke's sieving algorithm would be an interesting avenue for future work.

Another prominent theoretical approach to Diophantine finiteness is Kim's non-abelian Chabauty method~\cite{KIM2}. This method employs $p$-adic period maps and has achieved notable theoretical successes, including solving the Effective Faltings--Siegel Problem conditionally upon the Bloch–Kato conjecture. Successful unconditional implementations include computations for certain curves such as the famous "cursed curve", $X_{\mathrm{ns}}(13)/\mathbb{Q}$~\cite{BALAKRISHNAN_DOGRA}. However, state-of-the-art implementations of the Motivic Chabauty--Kim method~\cite{DCW,PLLG,REFINED_SELMER}, a particular variation of Kim's method geared towards solving the $S$ unit equation, remains practically constrained to sets $S$ of cardinality at most $2$, while our method effectively handles sets with cardinality $3$ within seconds.

The Lawrence--Venkatesh method~\cite{LV} offers an alternative $p$-adic analytic approach, combining aspects of Chabauty--Kim theory and Kodaira--Parshin families. Despite its algorithmic potential and significant theoretical contributions---reproving both Faltings' and Siegel's theorems explicitly---no practical computational implementation for explicit curves has yet emerged. A discussion of computational challenges in both Kim’s and Lawrence--Venkatesh’s methods is given by Balakrishnan~\cite{BALAKRISHNAN_SURVEY} et al.

Finally, Kantor~\cite{KANTORTHESIS} proposed a program to unify the Chabauty--Kim and Lawrence--Venkatesh methods into a common theoretical framework, subsequently expanded by Corwin and the author in~\cite{CKK_CZ}. While theoretically appealing, this unified approach has not yet resulted in practical computational implementations comparable to those described above.

\subsection{Plan of Paper}
\S\ref{chapter:modular_chabauty_for_emps} provides a rigorous description of our algorithm (see Algorithm~\ref{algorithm:mcemp}). \S\ref{section:effective_shafarevich} describes our approach to the Effective Shafarevich step, and \S\ref{section:fibre_computation} entails our approach to the Fibre Computation step. More specifically, we explain our universal approach in \S\ref{section:general_strategy}, and our moduli theoretic approach in \S\ref{section:modular_strategy}.

\S\ref{chapter:tables} records the numerical outputs of Algorithm~\ref{algorithm:mcemp} (runtime on the author's laptop, and point count) for the $S$-unit equation, and all modular curves of the form $\mathcal{Y}_1(N)$ with $N\in \{4,5,6,7,8,9,10,12\}$, subject to $\prod_{p\in S}p^2\lesssim 5\cdot 10^5$.

\section{Modular Chabauty for Elliptic Moduli Problems}
\label{chapter:modular_chabauty_for_emps}
In \S\ref{chapter:intro} we outlined the \emph{Modular Chabauty} paradigm: given a hyperbolic curve $\mathcal Y/\mathcal O_{K,S}$ endowed with a non-isotrivial abelian fibration $\pi: \mathcal{X}\to\mathcal{Y}$, one studies the image and fibres of the associated modular period map to control the set $\mathcal Y(\mathcal O_{K,S})$. The present chapter specialises this strategy to \emph{elliptic moduli problems} over $\mathbb Z[1/S]$ and turns it into an explicit algorithm.

Our main result is the following two--step procedure.

\begin{algorithm}[Modular Chabauty for elliptic moduli problems]
\label{algorithm:mcemp}
${}$\\
\begin{itemize}
    \item \textbf{Input:} an $S$--integral elliptic moduli problem $(\mathcal Y,\Phi_{\mathcal M})$.
    \item \textbf{Output:} the finite set $\mathcal Y(\mathbb Z[1/S])$.
\end{itemize}
The algorithm proceeds via:
\begin{enumerate}
  \item[1.] \textbf{Effective Shafarevich.} List all $\mathbb Q$--isomorphism classes of elliptic curves with good reduction outside $S$.
  \item[2.] \textbf{Fibre Computation.} For each curve $E$ from Step 1 compute the $S$--integral fibre $\Phi_{\mathcal M}^{-1}(E)(\mathbb Z[1/S])\subseteq\mathcal Y(\mathbb Z[1/S])$.
\end{enumerate}
\end{algorithm}

Step 1 (\S\ref{section:effective_shafarevich}) reduces to Cremona’s algorithms for modular elliptic curves and relies on the Modularity Theorem; this dependence is the main reason we restrict to $\mathbb Z[1/S]$ rather than $\mathcal O_{K,S}$. Step 2 (\S\ref{section:fibre_computation}) is carried out over an arbitrary ring of $S$-integers and is logically independent of the implementation of Step 1.

\begin{remark}
A Python/Sage implementation of Algorithm~\ref{algorithm:mcemp} may be found on~\cite{MY_IMPLEMENTATION}.
\end{remark}

\subsection{The Effective Shafarevich Step}
\label{section:effective_shafarevich}
The main reference here is Cremona’s foundational work~\cite{CREMONA}. Given a positive integer $d$, Cremona’s produces representatives of all $\mathbb{Q}$-isomorphism classes of elliptic curves of conductor~$d$ in two stages:
\begin{enumerate}
  \item[(i)] enumerate the $\mathbb Q$-isogeny classes of elliptic curves with conductor~$d$;
  \item[(ii)] within each class, list the distinct $\mathbb Q$\,--isomorphism classes.
\end{enumerate}
Stage\,(i) relies on the Modularity Theorem (Theorem~\ref{theorem:modularity}); stage\,(ii) employs Vélu’s formulas (Theorem~\ref{theorem:velu}) together with Mazur’s finiteness theorem for rational isogenies (Theorem~\ref{theorem:mazur_finiteness}). In practice we invoke Cremona’s algorithms only for the finitely many divisors of a single integer $N(S)$, defined next.

\begin{fact}[{\cite[p.~32]{VON_KANEL}}]\label{fact:finiteness_conductor}
Let $S$ be a finite set of rational primes and let $E/\mathbb Q$ be an elliptic curve with good reduction outside~$S$.  Writing $N_E$ for its conductor, one has $N_E\mid N(S)$ with
\[
  N(S)=c_2(S)\,c_3(S)\prod_{p\in S} p^{2}, \qquad
  c_2(S)=\begin{cases}8 & 2\in S,\\1 & 2\notin S,\end{cases}\;
  c_3(S)=\begin{cases}3 & 3\in S,\\1 & 3\notin S.\end{cases}
\]
\end{fact}
\begin{proof}
Because $E$ has good reduction outside~$S$, its conductor factors as $N_E=\prod_{p\in S}p^{f_p}$ with local exponents bounded by~\cite[IV\,10.4]{SILVERMAN_ELLIPTIC}
\[
  f_p\le \begin{cases}5 & p=2,\\3 & p=3,\\2 & p\ge 5.\end{cases}
\]
The claim follows immediately.
\end{proof}

Consequently, the Effective Shafarevich step reduces to applying Cremona’s routines for each $d\mid N(S)$.  The classical nature of those algorithms allows us merely to summarise them: §\ref{subsection:es_isogeny_list} recalls the enumeration of $\mathbb Q$-isogeny classes, while §\ref{subsection:es_class} describes the (finite) search within a fixed class.

An elementary combination of Fact~\ref{fact:finiteness_conductor} with the Modularity Theorem yields the following quantitative corollary, previously obtained by von Känel.

\begin{corollary}[{\cite[Thm.~7.1]{VON_KANEL}}]\label{corollary:finiteness}
For every integer $N\ge 1$ and finite set $S$ of primes one has
\[
  \#\mathcal Y_1(N)(\mathbb Z[1/S])\;\le\;\frac{2}{3}\,[\mathrm{SL}_2(\mathbb Z):\Gamma_1(N)]\,N(S)\prod_{p\in S}\Bigl(1+\frac1p\Bigr).
\]
\end{corollary}
\begin{remark}
The bound combines (i) the number of $\mathbb Q$\,--isogeny classes of conductor dividing $N(S)$—at most the genus of $X_0(d)$—together with Kenku’s bound of $8$ curves per isogeny class~\cite{KENKU}, and (ii) the fact that each elliptic curve admits at most $[\mathrm{SL}_2(\mathbb Z):\Gamma_1(N)]$ distinct $\Gamma_1(N)$\,--level structures.
\end{remark}

\subsubsection{Cremona’s Algorithm for Listing Isogeny Classes}
\label{subsection:es_isogeny_list}

We briefly recall Cremona’s method \cite[Ch.\,2]{CREMONA} for listing the $\mathbb{Q}$-isogeny classes of elliptic curves of a fixed conductor~$d$. Write $X_0(d)$ for the compactified coarse moduli space of pairs \((E,C)\) with a cyclic subgroup $C\subset E$ of order~$d$ (i.e.\ $\Gamma_0(d)$-level structure). A fundamental corollary of the Modularity Theorem is the following canonical bijection.

\begin{theorem}[Modularity Theorem {\cite{WILES, TAYLOR-WILES, BCDT}}]
\label{theorem:modularity}
For every positive integer \(d\) there is a canonical bijection
\[
\begin{array}{c}
\bigl\{\text{$\mathbb{Q}$-isogeny classes of elliptic curves of conductor $d$}\bigr\}\\[6pt]
\Longleftrightarrow\\[6pt]
\bigl\{\text{normalised integral, cuspidal weight-$2$ Hecke newforms on }X_0(d)\bigr\}.
\end{array}
\]
\end{theorem}

\begin{notation}
Denote by:    
\begin{itemize}
  \item $\mathcal I(d)$ -- the set of $\mathbb{Q}$-isogeny classes of elliptic curves of conductor $d$; for \(E/\mathbb{Q}\) write \([E]\in\mathcal I(d)\).
  \item $S_2(d)$ -- the $\mathbb{C}$-vector space of weight-$2$ cusp forms on~$X_0(d)$; $\mathcal H_2(d)^{\mathrm{new}}_{\mathbb{Z}}\subset S_2(d)$ -- the finite set of normalised integral Hecke newform (its cardinality is bounded above by $\dim_{\mathbb{C}} S_2(d) = g_0(d)$, the genus of \(X_0(d)\)).
\end{itemize}
\end{notation}

\smallskip
For \([E]\in\mathcal I(d)\) the corresponding newform \(f_E\) is obtained via the inverse Mellin transform of the completed $L$-function of \(E\). Conversely, for \(f\in\mathcal H_2(d)^{\mathrm{new}}_{\mathbb{Z}}\) one constructs
\[
  \Lambda_f \;=\; \Bigl\{\;\int_{\gamma} f(z)\,dz\;\Bigm|\;\gamma\in H_1\!\bigl(X_0(d),\mathbb{Z}\bigr)\Bigr\}, 
  \qquad E_f \;:=\; \mathbb{C}/\Lambda_f,
\]
and \(E_f\) descends to $\mathbb{Q}$. Cremona turns this correspondence into a practical procedure:

\begin{algorithm}[Cremona, step (1) of the Effective Shafarevich routine]
\label{algorithm:b1}
\leavevmode
\begin{description}
  \item[Input]  
    A conductor \(d\ge 1\); a Fourier-truncation parameter \(C\in\mathbb{Z}_{>0}\); a numerical error tolerance \(\varepsilon>0\).
  \item[Output]  
    A set \(\mathcal I(d)^{\mathrm{rep}}\) containing exactly one representative of each $\mathbb{Q}$-isogeny class of conductor~\(d\).
\end{description}

\begin{enumerate}
  \item Choose a \(\mathbb{Z}\)-basis \(\gamma_1,\dots,\gamma_{2g}\) of \(H_1(X_0(d),\mathbb{Z})\) with \(g=g_0(d)\).
  \item For each \(f=\sum_{n>0}a_nq^n\in\mathcal H_2(d)^{\mathrm{new}}_{\mathbb{Z}}\):
    \begin{enumerate}\renewcommand{\theenumii}{C\arabic{enumii}}
      \item Compute the first \(C\) Fourier coefficients \(a_1,\dots,a_C\) and form \(f_C=\sum_{n=1}^Ca_nq^n\).
      \item Integrate numerically to precision \(\varepsilon\):  
            \(\omega_{C,\varepsilon,i}=\int_{\gamma_i}f_C(z)\,dz\).
      \item Find generators \(\Omega_{C,\varepsilon,1},\Omega_{C,\varepsilon,2}\) of the lattice
            \(\Lambda_{f,C,\varepsilon}=\sum_{i}\mathbb{Z}\,\omega_{C,\varepsilon,i}\).
      \item Move \(\tau'=\Omega_{C,\varepsilon,1}/\Omega_{C,\varepsilon,2}\) into the standard fundamental domain
            \(\{\,z=x+iy\in\mathbb H : |x|\le\frac12,\,x^2+y^2\ge1\}\); denote the result by \(\tau\).
      \item Evaluate \(c_4(\tau),c_6(\tau)\) to precision \(\varepsilon\) and round to the nearest integers \(c_4^{C,\varepsilon},c_6^{C,\varepsilon}\in\mathbb{Z}\).
      \item Build an integral Weierstrass model \(E_f^{C,\varepsilon}\) from these invariants.
      \item Check equality of \(L\)-series.  
            If \(E_f^{C,\varepsilon}=E_f\), add \(E_f\) to \(\mathcal I(d)^{\mathrm{rep}}\);  
            otherwise return to step (C1) with higher precision.
    \end{enumerate}
\end{enumerate}
\end{algorithm}

\begin{remark}
Step (C1) is achieved by simultaneously diagonalizing the Hecke algebra with respect to the homology lattice, on which the Hecke action is given by \emph{explicit} integer matrices; see \cite[\S 2.7]{CREMONA}. This circumvents having to express the Hecke action on the space of weight $2$ cusp forms directly, on which we do not have such an explicit description.
\end{remark}

\begin{remark}
The rounding of the modular lattice invariants $c_4^{C,\epsilon}, c_6^{C,\epsilon}$ in step (C5) is justified by Edixhoven~\cite{EDIXHOVEN}, confirming that the exact invariants $c_4,c_6$ associated to $f$ are integers.
\end{remark}

\subsubsection{Listing all curves in a fixed isogeny class}
\label{subsection:es_class}
We now sketch the second stage of Cremona’s procedure \cite[Ch.\,3]{CREMONA}: given a single elliptic curve \(E/\mathbb{Q}\) of conductor~\(d\), list \emph{all} curves in its $\mathbb{Q}$-isogeny class. The method essentially reconstructs the \emph{isogeny graph} of~\(E\).

Since every pair of $\mathbb{Q}$-isogenous non-isomorphic rational elliptic curves are linked by a chain of cyclic prime degree (rational) isogenies, it is enough to determine, for every prime $\ell$, the set of rational elliptic curves $\ell$-isogenous to $E$. Mazur’s celebrated Rational Isogeny Theorem shows that only finitely many primes may occur:
\begin{theorem}[Mazur--Rational Isogenies {\cite[Theorem 1]{MAZUR_ISOGENY}}]
\label{theorem:mazur_finiteness}
Let $\ell$ denote a prime number, such that some rational elliptic curve admits a rational $\ell$-isogeny. Then \(\ell\in \{2,3,5,7,11,13,17,19,37,43,67,163\}\).
\end{theorem}
Hence every $\mathbb{Q}$-isogeny class is generated by cyclic isogenies of degrees in the above list. Moreover, Kenku showed~\cite{KENKU} that any such class contains at most eight $\mathbb{Q}$-isomorphism classes, so a breadth first search in the isogeny graph is bound to terminate.

\paragraph{Computing the \(\ell\)-isogenous quotients.}
It remains to explain how, for every pair $(E, \ell)$, consisting of an elliptic curve $E$ given by a Weierstrass model:
\[
  y^{2}+a_{1}xy+a_{3}y \;=\;
  x^{3}+a_{2}x^{2}+a_{4}x+a_{6},
\]
and a fixed prime \(\ell\) from Mazur’s list, one computes all \(\ell\)-isogenous curves \(E'\). 

If $\phi_{\ell}: E\to E'$ is an \(\ell\)-isogeny, then its kernel is a $\mathbb{Q}$-defined cyclic group, generated by a geometric point $P\in E(\overline{\mathbb{Q}})$ of exact order $\ell$. A necessary and sufficient condition for the quotient $E/\langle P \rangle$ to be $\mathbb{Q}$-defined is that the subgroup $\langle P \rangle$ is Galois-stable. Among the $\ell+1$ subgroups of $E[\ell]$, only two are stable under complex conjugation, hence descend to $\mathbb{R}$---a necessary condition for descent to $\mathbb{Q}$. Denote these two subgroups by $G_1$ and $G_2$.

Using the uniformization map associated with the period lattice of $E$ (already obtained in Algorithm~\ref{algorithm:b1}), one computes generators of $G_1$ and $G_2$. In order to obtain models for $E/G_i$, $i=1,2$, one uses Vélu’s formula:
\begin{theorem}[V\'{e}lu's Formula~\cite{VELU}]
\label{theorem:velu}
Let $P = (x,y) \in E(\overline{\mathbb{Q}})$ have exact order $\ell$. For $1 \le k \le \lfloor \ell/2 \rfloor$, put
\begin{equation}
t_k := 6x_k^2 + b_2x_k + b_4,\quad u_k := 4x_k^3 + b_2x_k^2 + 2b_4x_k + b_6.
\end{equation}
where $kP = (x_k, y_k)$, and \(b_2,b_4,b_6\) are the usual invariants\footnote{The invariants $b_2$, $b_4$, and $b_6$ are derived from rewriting the curve in the form $y^2 = 4x^3 + b_2x^2 + b_4x + b_6$, with $b_2 \in \{-5,\dots,6\}$ after minimising. See \cite[III.\S 1]{SILVERMAN_ELLIPTIC}.}. Let 
\begin{equation}
t = \sum_{k=1}^{\lfloor \ell/2 \rfloor} t_k,\quad w = \sum_{k=1}^{\lfloor \ell/2 \rfloor}(u_k + x_kt_k).    
\end{equation}
Then a model for $E/\langle P \rangle$ is given by:
\begin{equation}
y^{2}+a_{1}xy+a_{3}y = x^{3}+a_{2}x^{2}+(a_{4}-5t)x+(a_{6}-b_2t-7w).
\end{equation}
\end{theorem}
Note that while the points generating $G_i$ may not be individually rational, the subgroup $G_i$ itself can still be defined over $\mathbb{Q}$. When $G_i$ is $\mathbb{Q}$-defined, either the model produced by V\'{e}lu's formula, or the model obtain by the change of variables $(x, y) \mapsto (l^2x, l^3y)$ is integral. Rounding to the nearest integral coefficients, it remains to show that the resulting curve, $E_i$, is isogenous but not isomorphic.

To verify that $E_i$ is isogenous to $E$, one compares their $L$-functions. To verify that $E_i$ is non-isomorphic to $E$, one compares $j$-invariants. When $j(E)\neq j(E_i)$, the curves are necessarily non-isomorphic; otherwise $E_i$ is a twist of $E$, and one verifies that it is a non-trivial twist (see \cite[X.5.4]{SILVERMAN_ELLIPTIC}). For curves without complex multiplication this final check is superfluous, since such curves admit no non-trivial isogenous twists.

\subsection{The Fibre Computation Step}
\label{section:fibre_computation}
Once every elliptic curve \(E/K\) with good reduction outside \(S\) has been enumerated, the remaining task is to extract, for each such \(E\), the \(S\)-integral points of \(\mathcal Y\) lying above \(E\) under the modular period map. Concretely, we want an algorithm that, given

\begin{itemize}
  \item an \(S\)-integral elliptic moduli problem
        \(\bigl(\mathcal Y,\Phi_{\mathcal M}\bigr)/\mathcal O_{K,S}\), and
  \item an elliptic curve \(E/K\) with good reduction outside \(S\),
\end{itemize}
returns the finite set
\[
   \Phi_{\mathcal M}^{-1}(E)\bigl(\mathcal O_{K,S}\bigr)
     \;\subseteq\;
   \mathcal Y\bigl(\mathcal O_{K,S}\bigr),
\]
or a larger subset of $Y\bigl(\mathcal O_{K,S}\bigr)$.

We present two simple and practical strategies.
\begin{enumerate}
  \item \emph{A universal geometric approach.} Valid for \emph{any} \(S\)-integral elliptic moduli problem, it uses the algebraicity and quasi-finiteness of the \(j\)-map to reduce the fibre problem to solving a single univariate polynomial equation, followed by integrality checks.
  \item \emph{A moduli-theoretic approach (for modular curves).} A method adapted to modular curves such as \(\mathcal{Y}_1(N)\), which exploits the explicit moduli interpretation of $\mathcal{Y}_1(N)$, giving a model-independent description of the fibre, especially convenient for proving (non-)existence or even counting \(S\)-integral points.
\end{enumerate}

We now describe each method in turn.

\subsubsection{Fibre computation: a universal strategy}
\label{section:general_strategy}
Let \(\bigl((\mathcal Y,\Phi_{\mathcal M}),E\bigr)\) be as in the input specification detailed in the beginning of \S\ref{section:fibre_computation}. We give a purely geometric procedure that recovers a subset of $Y\bigl(\mathcal O_{K,S}\bigr)$ containing the fibre \(\Phi_{\mathcal M}^{-1}(E)(\mathcal O_{K,S})\), which applies to \emph{any} \(S\)-integral elliptic moduli problem.

\medskip
\paragraph{\textnormal{$j$}-map set-up.}
For \(Q\in\mathcal M_{1,1}(\mathcal O_{K,S})\) denote by
\(\mathcal E_{Q}\) the corresponding elliptic curve.
The \(j\)-invariant defines a morphism
\[
   j_{\mathcal M}\;:\;
   \mathcal M_{1,1}\;\longrightarrow\;\mathbb{A}^{1}_{\mathcal O_{K,S}},
   \qquad
   Q\;\longmapsto\;j\bigl(\mathcal E_{Q}\bigr).
\]
Compose with the modular period map to obtain  
\(j_{\mathcal Y}:=j_{\mathcal M}\circ\Phi_{\mathcal M}\);
this is a non-constant, quasi-finite morphism of affine schemes
\[
   j_{\mathcal Y}\;:\;\mathcal Y\;\longrightarrow\;
                \mathbb{A}^{1}_{\mathcal O_{K,S}}.
\]
By construction, for every $E\in \mathcal{M}_{1,1}(\mathcal{O}_{K,S})$ with $j$-invariant $j_E := j(E)$, $\Phi_{\mathcal{M}}^{-1}(E)(\mathcal{O}_{K,S})\subseteq j_{\mathcal Y}^{-1}(j_E)(\mathcal{O}_{K,S})\subseteq Y\bigl(\mathcal O_{K,S}\bigr)$.
\medskip
\begin{lemma}
Notations as above. There is an algorithm, which effectively computes the $S$-integral points of $j_{\mathcal Y}^{-1}(j_E)(\mathcal{O}_{K,S})$.
\end{lemma}
\begin{proof}
Put \(j_{E}:=j(E)\in\mathcal O_{K,S}\). The scheme-theoretic fibre \(j_{\mathcal Y}^{-1}(j_{E})\) is finite étale over \(\mathcal O_{K,S}\) and can be described explicitly:

\begin{itemize}
\item Present \(\mathcal Y\) as
      \(\mathrm{Spec}\,\!\bigl(\mathcal O_{K,S}[x_1,\dots,x_n]/I(\mathcal Y)\bigr)\).
\item Form the ideal
      \(J_E:=I(\mathcal Y)+(j_{\mathcal Y}-j_{E})\).
      The quotient algebra over \(\mathcal O_{K,S}\) defines the fibre.
\end{itemize}

Base-changing to \(K\) and running a Gröbner-basis elimination yields an
isomorphism
\[
   K[x_1,\dots,x_n]/J_E\;\otimes_{\mathcal O_{K,S}}K
      \;\cong\;
   K[t]/\bigl(f(t)\bigr)
\]
for some non-zero \(f(t)\in K[t]\).  
Thus \(K\)-rational points of the fibre correspond to roots of
\(f(t)=0\), which can be found effectively (e.g.\ \cite[\S3.6.2]{COHEN_BOOK}).
A root \(t_0\) gives a morphism
\(\mathrm{Spec}\, K\xrightarrow{t_0}j_{\mathcal Y}^{-1}(j_E)\).
If all coordinate functions \(x_i(t_0)\) lie in
\(\mathcal O_{K,S}\) then \(t_0\) factors through
\(\mathrm{Spec}\,\mathcal O_{K,S}\), producing an \(S\)-integral point of the fibre.    
\end{proof}

\begin{example}[thrice-punctured line]\label{example:level_structure_step}
Set \(\mathcal Y=\mathbb{P}^{1}\setminus\{0,1,\infty\}\) with coordinate \(t\).
The Legendre family
\(\mathcal E_t:y^{2}=x(x-1)(x-t)\) satisfies
\[
   j_{\mathcal Y}(t)=2^{8}\,\frac{(t^{2}-t+1)^{3}}{t^{2}(t-1)^{2}}.
\]
Given \(P\in\mathcal M_{1,1}(\mathbb{Z}[1/S])\) with associated curve
\(\mathcal E_{P}\) and \(j\)-invariant \(j_P\), the fibre
\(j_{\mathcal Y}^{-1}(j_P)\) is
\[
   \mathrm{Spec}\,\!\bigl(
     \mathcal O_{K,S}[t,1/t,1/(1-t)]/\bigl(j_{\mathcal Y}(t)-j_P\bigr)
   \bigr).
\]
Hence \(t\) must satisfy
\(
   2^{8}(t^{2}-t+1)^{3}=j_P\,t^{2}(t-1)^{2}.
\)
An $S$-integral solution \(t_0\) must also make \(t_0,\;1/t_0,\;1/(1-t_0)\) \(S\)-integral, i.e.\ \(t_0\) solves the $S$-unit equation.
\end{example}

\subsubsection{Fibre computation: a moduli-theoretic approach}
\label{section:modular_strategy}
We now specialise to the case \(\mathcal Y=\mathcal Y_1(N)\), the (affine) modular curve with \(\Gamma_1(N)\)-level structure, defined over \(\mathcal O_{K,S}\). Thanks to its moduli interpretation, the fibre of the period map admits a \emph{model--free} description, well suited to counting or proving (non-)existence of $S$-integral points.

\begin{proposition}\label{fact:modular_description}
There is a natural bijection between $\mathcal Y_1(N)\bigl(\mathcal O_{K,S}\bigr)$ and isomorphism classes of pairs $(E,P)$, where \(E/K\) is an elliptic curve with good reduction outside~\(S\) and \(P\in E(K)\) is a torsion point of exact order~\(N\).
\end{proposition}

\begin{proof}
The functor represented by \(\mathcal Y_1(N)\) sends a \(\mathbb{Z}[1/N]\)-scheme \(T\) to the set of isomorphism classes of pairs \((E\!\to\!T,\;P)\) with \(P\) a section of exact order~\(N\). Thus \(\mathcal Y_1(N)\bigl(\mathcal O_{K,S}\bigr)\) classifies pairs \((\mathcal E,\mathcal P)\) over \(\mathcal O_{K,S}\).

Because \(\mathcal Y_1(N)\) is separated, the inclusion \(\iota:\mathcal Y_1(N)(\mathcal O_{K,S})\hookrightarrow \mathcal Y_1(N)(K)\)
is injective. The image is contained in the set of pairs $(E,P)$, with $E/K$ an elliptic curve with good reduction outside $S$, and $P\in E[N](K)$, an exact order $N$ rational point. It remains to show that the inclusion map surjects onto such pairs.

Let \((E,P)\in\mathcal Y_1(N)(K)\) with \(E\) having good reduction outside~\(S\). Its Néron model \(\mathcal E/\mathcal O_{K,S}\) is an elliptic curve, and the Néron mapping property lifts \(P\) uniquely to a section \(\mathcal P\) of exact order~\(N\). Hence \((\mathcal E,\mathcal P)\) lies in \(\mathcal Y_1(N)(\mathcal O_{K,S})\), and \(\iota(\mathcal E,\mathcal P)=(E,P)\), proving surjectivity.
\end{proof}

\paragraph{Computational consequences.}
For a fixed elliptic curve \(E/K\) with good reduction outside $S$, the fibre \(\Phi_{\mathcal M}^{-1}(E)\bigl(\mathcal O_{K,S}\bigr)\) is identified with the set of exact order-\(N\) rational torsion points on \(E\). Hence the fibre computation problem reduces to computing \(E[N](K)\), and retaining the subset of points of exact order~\(N\). We recall an efficient division-polynomial algorithm for this task in~\S\ref{subsection:division_polynomials}.

\begin{remark}
The same argument extends to other level structures (e.g.\ full level~\(N\)). We restrict to \(\Gamma_1(N)\) because a \(\mathbb{Q}\)-rational elliptic curve has no non-trivial full level-\(N\) structure for \(N>2\) by Mazur’s Torsion Theorem~\cite[Theorem 2]{MAZUR_ISOGENY}.
\end{remark}

\paragraph{A moduli-theoretic view of \texorpdfstring{$\mathbb{P}^{1}\!\setminus\!\{0,1,\infty\}$}{P1-0,1,∞}}
The functor attached to the subgroup \(\Gamma(2)\), which classifies pairs \((E,\alpha)\) with a \emph{full level-\(2\)} structure \(\alpha:E[2]\xrightarrow{\sim}(\mathbb{Z}/2\mathbb{Z})^{2}\), is not representable by a scheme but \emph{is} represented by a Deligne--Mumford stack \(\mathcal Y(2)\). Its coarse moduli space is the thrice--punctured line \(\mathbb{P}^{1}\setminus\{0,1,\infty\}\); hence, endowing the set $\mathbb{P}^{1}\setminus\{0,1,\infty\}(\mathcal{O}_{K,S})$ with a clean modular description, analogous to that of $\mathcal{Y}_{1}(N)(\mathcal{O}_{K,S})$.

\begin{fact}
\label{fact:level2_description}
There is a canonical bijection
\[
   \mathbb{P}^{1}\!\setminus\!\{0,1,\infty\}
        \bigl(\mathcal O_{K,S}\bigr)
   \;\longleftrightarrow\;
   \Bigl\{\text{isomorphism classes of }(E,\alpha)\Bigr\},
\]
where \(E/K\) has good reduction outside \(S\) and \(\alpha\) is a full level-$2$ structure on $E$. Two pairs \((E,\alpha)\) and \((E',\alpha')\) are identified if they become isomorphic over the maximal extension \(K_S/K\) unramified outside~\(S\).
\end{fact}

In practice, two curves having good reduction outside \(S\) become isomorphic over \(K_S\) precisely when their \(j\)-invariants coincide, so enumerating \(S\)-integral points is tantamount to:
\smallskip
\begin{enumerate}
  \item listing the \(j\)-invariants of elliptic curves \(E/K\) (good outside~\(S\)) that admit a full level-\(2\) structure;
  \item for each such \(E\), counting its distinct level-\(2\) structures.
\end{enumerate}

A generic elliptic curve with full level-\(2\) structure admits \(\lvert\mathrm{GL}_{2}(\mathbb{F}_{2})\rvert = 6\) such structures. If \(j(E)=0\) or \(1728\) the automorphism group of \(E\) is larger, and the count drops to \(2\) or \(3\), respectively.

\subsubsection{Division polynomials}
\label{subsection:division_polynomials}
To compute (or count) the \(K\)-rational torsion points of exact order \(N\) on an elliptic curve \(E/K\) one may use the \emph{division–polynomial method}. We recall its essentials.

\begin{definition}
For the short Weierstrass model \(E:y^{2}=x^{3}+Ax+B\) (with \(A,B\in\mathbb{Z}\)) the \emph{division polynomials}
\(\{\psi_{n}\}_{n\ge0}\subset\mathbb{Z}[x,y,A,B]\) are defined recursively by:
\[
\psi_{0}=0,\quad\psi_{1}=1,\quad\psi_{2}=2y,
\]
\[
\psi_{3}=3x^{4}+6Ax^{2}+12Bx-A^{2},
\]
\[
\psi_{4}=4y\bigl(x^{6}+5Ax^{4}+20Bx^{3}-5A^{2}x^{2}-4ABx-8B^{2}-A^{3}\bigr),
\]
and for \(m\ge2\)
\[
  \psi_{2m+1}
    =\psi_{m+2}\,\psi_{m}^{3}-\psi_{m-1}\,\psi_{m+1}^{3},
\qquad
  \psi_{2m}
    =\frac{\psi_{m}}{2y}\bigl(\psi_{m+2}\,\psi_{m-1}^{2}
                              -\psi_{m-2}\,\psi_{m+1}^{2}\bigr)
    \quad(m\ge3).
\]    
\end{definition}

The key property of the $m$'th division polynomials is summarized by the following theorem.
\begin{theorem}
\label{thm:division_poly_roots}
Let \(E:y^{2}=x^{3}+Ax+B\) as above. For each $m\ge 1$:
\begin{enumerate}
\item The roots of \(\psi_{2m-1}\) are the \(x\)-coordinates of the non-trivial points in \(E[2m+1]\).
\item The roots of \(\psi_{2m}/y\) are the \(x\)-coordinates of the points in \(E[2m]\setminus E[2]\); note that \(\psi_{2m}\in y\,\mathbb{Z}[x,A,B]\).
\end{enumerate}
\end{theorem}

To determine the exact order $N$ rational torsion points of $E/K$, compute \(\psi_{N}\) (or, if \(N\) is even, \(\psi_{N}/y\)), and compute its $K$-rational roots to obtain the \(x\)-coordinates of the geometric \(N\)-torsion points. For each rational $x$-coordinate, recover its \(y\)-coordinate pair(s) via the Weierstrass equation, and collect the \(K\)-rational points in \(E[N](K)\). Then, one can verify that a rational $N$-torsion point $P$ has exact order $N$ either by direct computation using the group law, or by demonstrating that for every proper divisor $d|N$, $\psi_d(P)\neq 0$.

\section{Computational Tables}
\label{chapter:tables}
This chapter compiles the numerical output of Algorithm~\ref{algorithm:mcemp} into tables. Each table records, for a fixed moduli problem and a selection of finite sets \(S\) of primes, subject to $\prod_{p\in S}p^2\lesssim 5\cdot 10^5$, two basic statistics:
\begin{itemize}
  \item the running time (in seconds) on the author’s laptop, and
  \item the cardinality of the corresponding set of \(S\)-integral points.
\end{itemize}
In \S~\ref{section:modular_tables} we exhibit output tables for modular curves with $\Gamma_1(N)$ level structure for each $N\in \{4,5,6,7,8,9,10,12\}$, and in \S~\ref{section:s_unit_tables} we exhibit similar data for the $S$-unit equation, with one exception: in Table~\ref{table:s_unit_small}, we also record the $j$-invariants of rational elliptic curves with good reduction outside $S$ admitting full level $2$-structure. 

\subsection{$S$-Unit Equation}
\label{section:s_unit_tables}
The tables below list running times and solution counts for the \(S\)-unit equation \(u+v=1,\;u,v\in\mathbb Z[1/S]^{\times}\), for all sets $S$ containing the prime $2$, subject to $\prod_{p\in S}p^2 \lesssim 5\cdot 10^5$. Table~\ref{table:s_unit_small} below, also records the $j$-invariants of rational elliptic curves with good reduction outside $S$, which admit a full level $2$-structure.

\begin{table}[H]
    \centering
    \caption{Summary of Algorithm Outputs for Solving the \(S\)-unit Equation ($|S| \le 2$)}
    \label{table:s_unit_small}
    \begin{tabular}{@{}lccc@{}}
        \toprule
        \(S\) & Time (sec) & \(j\)-invariants & \(\#\mathbb{P}^1\setminus\{0,1,\infty\}(\mathbb{Z}[1/S])\) \\ 
        \midrule
        \(\{2\}\)       & 0.0165 & \(\{1728\}\) & 3 \\
        \(\{2,3\}\)     & 0.2187  & \(\left\{\frac{35152}{9},\, \frac{1556068}{81},\, \frac{21952}{9},\, 1728\right\}\) & 21 \\
        \(\{2,5\}\)     & 0.1770  & \(\{1728,\, \frac{148176}{25}\}\) & 9 \\
        \(\{2, 7\}\)    & 0.2318 & \(\{1728,\ \frac{740772}{49}\}\) & 9 \\
        \(\{2, 11\}\)   & 0.1513 & \(\{1728\}\)                   & 3 \\
        \(\{2, 13\}\)   & 0.1624 & \(\{1728\}\)                   & 3 \\
        \(\{2, 17\}\)   & 0.1994  & \(\{1728,\ \frac{20346417}{289}\}\) & 9 \\
        \(\{2, 19\}\)   & 0.1872 & \(\{1728\}\)                   & 3 \\
        \(\{2, 23\}\)   & 0.1561 & \(\{1728\}\)                   & 3 \\
        \(\{2, 29\}\)   & 0.1627 & \(\{1728\}\)                   & 3 \\
        \(\{2, 31\}\)   & 0.2107  & \(\{1728,\ \frac{979146657}{3844}\}\) & 9 \\
        \(\{2, 37\}\)   & 0.1719 & \(\{1728\}\)                   & 3 \\
        \(\{2, 41\}\)   & 0.1113  & \(\{1728\}\)                   & 3 \\
        \(\{2, 43\}\)   & 0.1138 & \(\{1728\}\)                   & 3 \\
        \(\{2, 47\}\)   & 0.0731 & \(\{1728\}\)                   & 3 \\
        \(\{2, 53\}\)   & 0.1465 & \(\{1728\}\)                   & 3 \\
        \(\{2, 59\}\)   & 0.1949 & \(\{1728\}\)                   & 3 \\
        \(\{2, 61\}\)   & 0.0817 & \(\{1728\}\)                   & 3 \\
        \(\{2, 67\}\)   & 0.0684 & \(\{1728\}\)                   & 3 \\
        \(\{2, 71\}\)   & 0.1205 & \(\{1728\}\)                   & 3 \\
        \(\{2, 73\}\)   & 0.0812 & \(\{1728\}\)                   & 3 \\
        \(\{2, 79\}\)   & 0.1706 & \(\{1728\}\)                   & 3 \\
        \(\{2, 83\}\)   & 0.1159 & \(\{1728\}\)                   & 3 \\
        \(\{2, 89\}\)   & 0.1331 & \(\{1728\}\)                   & 3 \\
        \(\{2, 97\}\)   & 0.1924  & \(\{1728\}\)                   & 3 \\
        \(\{2, 101\}\)  & 0.1406 & \(\{1728\}\)                   & 3 \\
        \(\{2, 103\}\)  & 0.0624 & \(\{1728\}\)                   & 3 \\
        \(\{2, 107\}\)  & 0.1606 & \(\{1728\}\)                   & 3 \\
        \(\{2, 109\}\)  & 0.1118 & \(\{1728\}\)                   & 3 \\
        \(\{2, 113\}\)  & 0.1166 & \(\{1728\}\)                   & 3 \\
        \bottomrule
    \end{tabular}
\end{table}

\begin{table}[H]
    \centering
    \caption{Summary of Algorithm Outputs for $|S| = 3$}
    \label{tab:algorithm_outputs}
    \begin{tabular}{@{}lcc@{}}
        \toprule
        \(S\) & Time (sec) & \(\#\mathbb{P}^1\setminus\{0,1,\infty\}(\mathbb{Z}[1/S])\) \\ 
        \midrule
        \(\{2, 3, 5\}\)    & 2.6472 & 99 \\
        \(\{2, 3, 7\}\)    & 3.2588 & 75 \\
        \(\{2, 3, 11\}\)   & 2.5991 & 57 \\
        \(\{2, 3, 13\}\)   & 2.4983 & 51 \\
        \(\{2, 3, 17\}\)   & 2.6254 & 51 \\
        \(\{2, 3, 19\}\)   & 2.4208 & 45 \\
        \(\{2, 3, 23\}\)   & 2.3226 & 39 \\
        \(\{2, 5, 7\}\)    & 1.8417 & 33 \\
        \(\{2, 5, 11\}\)   & 1.9373 & 27 \\
        \(\{2, 5, 13\}\)   & 1.8619 & 27 \\
        \(\{2, 5, 17\}\)   & 1.7077 & 21 \\
        \(\{2, 5, 19\}\)   & 1.5106 & 15 \\
        \(\{2, 5, 23\}\)   & 1.3175 & 15 \\
        \(\{2, 7, 11\}\)   & 1.7789 & 21 \\
        \(\{2, 7, 13\}\)   & 1.6810 & 21 \\
        \(\{2, 7, 17\}\)   & 1.5792 & 21 \\
        \bottomrule
    \end{tabular}
\end{table}

\subsection{Modular Curves with $\Gamma_1(N)$-Level Structure}
\label{section:modular_tables}
The following tables list the run time of Algorithm~\ref{algorithm:mcemp}, and the cardinality of $S$-integral points on all curves $\mathcal{Y}_1(N)/\mathbb{Z}[1/S]$, with \(N\in\{4,5,6,7,8,9,10,12\}\), and finite sets $S$ containing all of $N$'s prime divisors, subject to $\prod_{p\in S}p^2\lesssim 5\cdot 10^5$.

\begin{table}[H]
\caption{Summary of Algorithm Outputs for \(N = 4,5\)}
\centering
    \begin{minipage}[t]{0.45\textwidth}
        \centering
        \label{tab:algorithm_outputs1}
        \begin{tabular}{@{}lcc@{}}
            \toprule
            \(S\) & Time (sec) & \(\#\mathcal{Y}_1(4)(\mathbb{Z}[1/S])\) \\ 
            \midrule
            \(\{2\}\)      & 0.0135 & 4 \\
            \(\{2, 3\}\)   & 0.1837 & 32 \\
            \(\{2, 5\}\)   & 0.1735 & 16 \\
            \(\{2, 7\}\)   & 0.1979 & 12 \\
            \(\{2, 11\}\)  & 0.1768 & 4 \\
            \(\{2, 13\}\)  & 0.1401 & 4 \\
            \(\{2, 17\}\)  & 0.2386 & 16 \\
            \(\{2, 19\}\)  & 0.1581 & 4 \\
            \(\{2, 23\}\)  & 0.1416 & 4 \\
            \(\{2, 29\}\)  & 0.1478 & 4 \\
            \(\{2, 31\}\)  & 0.1840 & 12 \\
            \(\{2, 37\}\)  & 0.2043 & 4 \\
            \(\{2, 41\}\)  & 0.1331 & 4 \\
            \(\{2, 43\}\)  & 0.1621 & 4 \\
            \(\{2, 47\}\)  & 0.0852 & 4 \\
            \(\{2, 53\}\)  & 0.1724 & 4 \\
            \(\{2, 59\}\)  & 0.3073 & 4 \\
            \(\{2, 61\}\)  & 0.0713 & 4 \\
            \(\{2, 67\}\)  & 0.0763 & 4 \\
            \(\{2, 71\}\)  & 0.1367 & 4 \\
            \(\{2, 73\}\)  & 0.0933 & 4 \\
            \(\{2, 79\}\)  & 0.2258 & 4 \\
            \(\{2, 83\}\)  & 0.1053 & 4 \\
            \(\{2, 89\}\)  & 0.1347 & 4 \\
            \(\{2, 97\}\)  & 0.1027 & 4 \\
            \(\{2, 101\}\) & 0.1819 & 4 \\
            \(\{2, 103\}\) & 0.0622 & 4 \\
            \(\{2, 107\}\) & 0.1738 & 4 \\
            \(\{2, 109\}\) & 0.1072 & 4 \\
            \(\{2, 113\}\) & 0.1161 & 4 \\
            \(\{2, 3, 5\}\)   & 2.4080 & 156 \\
            \(\{2, 3, 7\}\)   & 2.4869 & 124 \\
            \(\{2, 3, 11\}\)  & 2.4469 & 88  \\
            \(\{2, 3, 13\}\)  & 2.3311 & 88  \\
            \(\{2, 3, 17\}\)  & 2.5106 & 80  \\
            \(\{2, 3, 19\}\)  & 2.2555 & 68  \\
            \(\{2, 3, 23\}\)  & 2.2155 & 60  \\
            \(\{2, 5, 7\}\)   & 1.9311 & 48  \\
            \(\{2, 5, 11\}\)  & 1.7349 & 48  \\
            \(\{2, 5, 13\}\)  & 1.7702 & 44  \\
            \(\{2, 5, 17\}\)  & 1.6994 & 36  \\
            \(\{2, 5, 19\}\)  & 1.3384 & 28  \\
            \(\{2, 5, 23\}\)  & 1.2283 & 24  \\
            \(\{2, 7, 11\}\)  & 1.7588 & 32  \\
            \(\{2, 7, 13\}\)  & 1.7050 & 28  \\
            \(\{2, 7, 17\}\)  & 1.4874 & 32  \\
            \bottomrule
        \end{tabular}
    \end{minipage}\hfill
    \begin{minipage}[t]{0.45\textwidth}
        \centering
        \label{tab:algorithm_outputs4}
        \begin{tabular}{@{}lcc@{}}
            \toprule
            \(S\) & Time (sec) & \(\#\mathcal{Y}_1(5)(\mathbb{Z}[1/S])\) \\ 
            \midrule
            \(\{5\}\)        & 0.0001 & 0 \\
            \(\{2, 5\}\)     & 0.2007 & 8 \\
            \(\{3, 5\}\)     & 0.1420 & 4 \\
            \(\{5, 7\}\)     & 0.0855 & 4 \\
            \(\{5, 11\}\)    & 0.0754 & 8 \\
            \(\{5, 13\}\)    & 0.0730 & 4 \\
            \(\{5, 17\}\)    & 0.0770 & 0 \\
            \(\{5, 19\}\)    & 0.0867 & 0 \\
            \(\{5, 23\}\)    & 0.0418 & 0 \\
            \(\{5, 29\}\)    & 0.0221 & 0 \\
            \(\{5, 31\}\)    & 0.0658 & 4 \\
            \(\{5, 37\}\)    & 0.1051 & 0 \\
            \(\{5, 41\}\)    & 0.0994 & 0 \\
            \(\{5, 43\}\)    & 0.0497 & 0 \\
            \(\{5, 47\}\)    & 0.0313 & 0 \\
            \(\{5, 53\}\)    & 0.0498 & 0 \\
            \(\{5, 59\}\)    & 0.0345 & 0 \\
            \(\{5, 61\}\)    & 0.0634 & 0 \\
            \(\{5, 67\}\)    & 0.0425 & 0 \\
            \(\{5, 71\}\)    & 0.0308 & 0 \\
            \(\{5, 73\}\)    & 0.0200 & 0 \\
            \(\{5, 79\}\)    & 0.1138 & 4 \\
            \(\{5, 83\}\)    & 0.0281 & 0 \\
            \(\{5, 89\}\)    & 0.0415 & 0 \\
            \(\{5, 97\}\)    & 0.0330 & 0 \\
            \(\{5, 101\}\)   & 0.0326 & 0 \\
            \(\{5, 103\}\)   & 0.0132 & 0 \\
            \(\{5, 107\}\)   & 0.0111 & 0 \\
            \(\{5, 109\}\)   & 0.0808 & 0 \\
            \(\{5, 113\}\)   & 0.0272 & 0 \\
            \(\{5, 127\}\)   & 0.0228 & 0 \\
            \(\{5, 131\}\)   & 0.0105 & 0 \\
            \(\{5, 137\}\)   & 0.0330 & 0 \\
            \(\{5, 139\}\)   & 0.0294 & 0 \\
            \(\{2, 3, 5\}\)     & 2.4983 & 20 \\
            \(\{2, 5, 7\}\)     & 1.7387 & 12 \\
            \(\{2, 5, 11\}\)    & 1.6686 & 28 \\
            \(\{2, 5, 13\}\)    & 1.8241 & 12 \\
            \(\{2, 5, 17\}\)    & 1.5994 & 8  \\
            \(\{2, 5, 19\}\)    & 1.3286 & 12 \\
            \(\{2, 5, 23\}\)    & 1.2272 & 8  \\
            \(\{3, 5, 7\}\)     & 0.9254 & 8  \\
            \(\{3, 5, 11\}\)    & 0.9940 & 12 \\
            \(\{3, 5, 13\}\)    & 1.1061 & 8  \\
            \(\{3, 5, 17\}\)    & 0.8919 & 4  \\
            \(\{3, 5, 19\}\)    & 0.9392 & 8  \\
            \(\{3, 5, 23\}\)    & 1.0129 & 4  \\
            \(\{5, 7, 11\}\)    & 0.7529 & 12 \\
            \(\{5, 7, 13\}\)    & 0.8013 & 8  \\
            \(\{5, 7, 17\}\)    & 0.5915 & 4  \\
            \(\{5, 7, 19\}\)    & 0.6266 & 8  \\
            \bottomrule
        \end{tabular}
    \end{minipage}
\end{table}

\begin{table}[H]
\caption{Summary of Algorithm Outputs for \(N = 6, 7, 10, 12\)}
\centering
    \begin{minipage}[t]{0.45\textwidth}
        \centering
        \label{tab:algorithm_outputs2}
        \begin{tabular}{@{}lcc@{}}
            \toprule
            \(S\) & Time (sec) & \(\#\mathcal{Y}_1(6)(\mathbb{Z}[1/S])\) \\ 
            \midrule
            \(\{2, 3\}\)       & 0.1864 & 4  \\
            \(\{2, 3, 5\}\)    & 3.9033 & 44 \\
            \(\{2, 3, 7\}\)    & 2.6826 & 28 \\
            \(\{2, 3, 11\}\)   & 2.5230 & 20 \\
            \(\{2, 3, 13\}\)   & 2.7724 & 12 \\
            \(\{2, 3, 17\}\)   & 2.5240 & 20 \\
            \(\{2, 3, 19\}\)   & 2.4760 & 12 \\
            \(\{2, 3, 23\}\)   & 2.3740 & 12 \\
            \bottomrule
        \end{tabular}
    \vspace*{1em}\par
    \begin{tabular}{@{}lcc@{}}
            \toprule
            \(S\) & Time (sec) & \(\#\mathcal{Y}_1(10)(\mathbb{Z}[1/S])\) \\ 
            \midrule
            \(\{2, 5\}\)        & 0.1568 & 0 \\
            \(\{2, 3, 5\}\)     & 2.5255 & 8 \\
            \(\{2, 5, 7\}\)     & 1.8679 & 0 \\
            \(\{2, 5, 11\}\)    & 1.6716 & 0 \\
            \(\{2, 5, 13\}\)    & 1.7386 & 0 \\
            \(\{2, 5, 17\}\)    & 1.6767 & 0 \\
            \(\{2, 5, 19\}\)    & 1.3491 & 0 \\
            \(\{2, 5, 23\}\)    & 1.2131 & 0 \\
            \bottomrule
    \end{tabular}
    \vspace*{1em}\par
    \begin{tabular}{@{}lcc@{}}
            \toprule
            \(S\) & Time (sec) & \(\#\mathcal{Y}_1(12)(\mathbb{Z}[1/S])\) \\ 
            \midrule
            \(\{2, 3\}\)       & 0.1832 & 0 \\
            \(\{2, 3, 5\}\)    & 2.4209 & 4 \\
            \(\{2, 3, 7\}\)    & 2.6258 & 0 \\
            \(\{2, 3, 11\}\)   & 2.3782 & 0 \\
            \(\{2, 3, 13\}\)   & 2.4077 & 0 \\
            \(\{2, 3, 17\}\)   & 2.3736 & 0 \\
            \(\{2, 3, 19\}\)   & 2.2248 & 0 \\
            \(\{2, 3, 23\}\)   & 2.2322 & 0 \\
            \bottomrule
    \end{tabular}
    \end{minipage}\hfill
    \begin{minipage}[t]{0.45\textwidth}
        \centering
        \label{tab:algorithm_outputs2}
        \begin{tabular}{@{}lcc@{}}
            \toprule
            \(S\) & Time (sec) & \(\#\mathcal{Y}_1(7)(\mathbb{Z}[1/S])\) \\ 
            \midrule
            \(\{7\}\)           & 0.0103 & 0 \\
            \(\{2, 7\}\)        & 0.2185 & 0 \\
            \(\{3, 7\}\)        & 0.1433 & 0 \\
            \(\{5, 7\}\)        & 0.0947 & 0 \\
            \(\{7, 11\}\)       & 0.0968 & 0 \\
            \(\{7, 13\}\)       & 0.0791 & 0 \\
            \(\{7, 17\}\)       & 0.0558 & 0 \\
            \(\{7, 19\}\)       & 0.0475 & 0 \\
            \(\{7, 23\}\)       & 0.0542 & 0 \\
            \(\{7, 29\}\)       & 0.0781 & 0 \\
            \(\{7, 31\}\)       & 0.0295 & 0 \\
            \(\{7, 37\}\)       & 0.0895 & 0 \\
            \(\{7, 41\}\)       & 0.0129 & 0 \\
            \(\{7, 43\}\)       & 0.0362 & 0 \\
            \(\{7, 47\}\)       & 0.0298 & 0 \\
            \(\{7, 53\}\)       & 0.0518 & 0 \\
            \(\{7, 59\}\)       & 0.0199 & 0 \\
            \(\{7, 61\}\)       & 0.0418 & 0 \\
            \(\{7, 67\}\)       & 0.0566 & 0 \\
            \(\{7, 71\}\)       & 0.0800 & 0 \\
            \(\{7, 73\}\)       & 0.0435 & 0 \\
            \(\{7, 79\}\)       & 0.0256 & 0 \\
            \(\{7, 83\}\)       & 0.0289 & 0 \\
            \(\{7, 89\}\)       & 0.0523 & 0 \\
            \(\{7, 97\}\)       & 0.0298 & 0 \\
            \(\{7, 101\}\)      & 0.0517 & 0 \\
            \(\{2, 3, 7\}\)    & 2.4767 & 6 \\
            \(\{2, 5, 7\}\)    & 1.7123 & 6 \\
            \(\{2, 7, 11\}\)    & 1.7302 & 0 \\
            \(\{2, 7, 13\}\)    & 1.5428 & 6 \\
            \(\{2, 7, 17\}\)    & 1.4337 & 0 \\
            \(\{3, 5, 7\}\)     & 0.9462 & 0 \\
            \(\{3, 7, 11\}\)    & 1.0337 & 0 \\
            \(\{3, 7, 13\}\)    & 1.5428 & 0 \\
            \(\{3, 7, 17\}\)    & 0.9534 & 0 \\
            \(\{3, 7, 19\}\)    & 0.9851 & 0 \\
            \(\{5, 7, 11\}\)    & 0.7375 & 0 \\
            \(\{5, 7, 13\}\)    & 0.6591 & 0 \\
            \(\{5, 7, 17\}\)    & 0.6190 & 0 \\
            \(\{5, 7, 19\}\)    & 0.6282 & 0 \\
            \bottomrule
        \end{tabular}
    \end{minipage}
\end{table}

\begin{table}[H]
\caption{Summary of Algorithm Outputs for \(N = 8,9\)}
\centering
    \begin{minipage}[t]{0.45\textwidth}
        \centering
        \label{tab:algorithm_outputs3}
        \begin{tabular}{@{}lcc@{}}
            \toprule
            \(S\) & Time (sec) & \(\#\mathcal{Y}_1(8)(\mathbb{Z}[1/S])\) \\ 
            \midrule
            \(\{2\}\)          & 0.0139 & 0 \\
            \(\{2, 3\}\)       & 0.2318 & 4 \\
            \(\{2, 5\}\)       & 0.2286 & 0 \\
            \(\{2, 7\}\)       & 0.2362 & 0 \\
            \(\{2, 11\}\)      & 0.1692 & 0 \\
            \(\{2, 13\}\)      & 0.1452 & 0 \\
            \(\{2, 17\}\)      & 0.2642 & 0 \\
            \(\{2, 19\}\)      & 0.1971 & 0 \\
            \(\{2, 23\}\)      & 0.1504 & 0 \\
            \(\{2, 29\}\)      & 0.1696 & 0 \\
            \(\{2, 31\}\)      & 0.2143 & 0 \\
            \(\{2, 37\}\)      & 0.3617 & 0 \\
            \(\{2, 41\}\)      & 0.1448 & 0 \\
            \(\{2, 43\}\)      & 0.1462 & 0 \\
            \(\{2, 47\}\)      & 0.0741 & 0 \\
            \(\{2, 53\}\)      & 0.1823 & 0 \\
            \(\{2, 59\}\)      & 0.2329 & 0 \\
            \(\{2, 61\}\)      & 0.1113 & 0 \\
            \(\{2, 67\}\)      & 0.0695 & 0 \\
            \(\{2, 71\}\)      & 0.1356 & 0 \\
            \(\{2, 73\}\)      & 0.1225 & 0 \\
            \(\{2, 79\}\)      & 0.1799 & 0 \\
            \(\{2, 83\}\)      & 0.1212 & 0 \\
            \(\{2, 89\}\)      & 0.1456 & 0 \\
            \(\{2, 97\}\)      & 0.0998 & 0 \\
            \(\{2, 101\}\)     & 0.1411 & 0 \\
            \(\{2, 103\}\)     & 0.0724 & 0 \\
            \(\{2, 107\}\)     & 0.1446 & 0 \\
            \(\{2, 109\}\)     & 0.0860 & 0 \\
            \(\{2, 113\}\)     & 0.0943 & 0 \\
            \(\{2, 3, 5\}\)     & 2.5480 & 8 \\
            \(\{2, 3, 7\}\)     & 2.8088 & 16 \\
            \(\{2, 3, 11\}\)    & 3.1591 & 4 \\
            \(\{2, 3, 13\}\)    & 2.4920 & 4 \\
            \(\{2, 3, 17\}\)    & 2.4770 & 8 \\
            \(\{2, 3, 19\}\)    & 2.2460 & 4 \\
            \(\{2, 3, 23\}\)    & 2.2481 & 4 \\
            \(\{2, 5, 7\}\)     & 1.6977 & 0 \\
            \(\{2, 5, 11\}\)    & 2.7450 & 0 \\
            \(\{2, 5, 13\}\)    & 1.7515 & 0 \\
            \(\{2, 5, 17\}\)    & 1.6006 & 0 \\
            \(\{2, 5, 19\}\)    & 1.3393 & 0 \\
            \(\{2, 5, 23\}\)    & 1.3234 & 0 \\
            \(\{2, 7, 11\}\)    & 1.6221 & 0 \\
            \(\{2, 7, 13\}\)    & 1.5486 & 0 \\
            \(\{2, 7, 17\}\)    & 1.4325 & 0 \\
            \bottomrule
        \end{tabular}
    \end{minipage}\hfill
    \begin{minipage}[t]{0.45\textwidth}
        \centering
        \label{tab:algorithm_outputs2}
        \begin{tabular}{@{}lcc@{}}
            \toprule
            \(S\) & Time (sec) & \(\#\mathcal{Y}_1(9)(\mathbb{Z}[1/S])\) \\ 
            \midrule
            \(\{3\}\)           & 0.0026 & 0 \\
            \(\{2, 3\}\)        & 0.1818 & 6 \\
            \(\{3, 5\}\)        & 0.1111 & 0 \\
            \(\{3, 7\}\)        & 0.1434 & 0 \\
            \(\{3, 11\}\)       & 0.1251 & 0 \\
            \(\{3, 13\}\)       & 0.0720 & 0 \\
            \(\{3, 17\}\)       & 0.1451 & 0 \\
            \(\{3, 19\}\)       & 0.1230 & 0 \\
            \(\{3, 23\}\)       & 0.1664 & 0 \\
            \(\{3, 29\}\)       & 0.0246 & 0 \\
            \(\{3, 31\}\)       & 0.0280 & 0 \\
            \(\{3, 37\}\)       & 0.1279 & 0 \\
            \(\{3, 41\}\)       & 0.0643 & 0 \\
            \(\{3, 43\}\)       & 0.1156 & 0 \\
            \(\{3, 47\}\)       & 0.1360 & 0 \\
            \(\{3, 53\}\)       & 0.0469 & 0 \\
            \(\{3, 59\}\)       & 0.0270 & 0 \\
            \(\{3, 61\}\)       & 0.0845 & 0 \\
            \(\{3, 67\}\)       & 0.1113 & 0 \\
            \(\{3, 71\}\)       & 0.0645 & 0 \\
            \(\{3, 73\}\)       & 0.1508 & 0 \\
            \(\{3, 79\}\)       & 0.0467 & 0 \\
            \(\{3, 83\}\)       & 0.0827 & 0 \\
            \(\{3, 89\}\)       & 0.1121 & 0 \\
            \(\{3, 97\}\)       & 0.1302 & 0 \\
            \(\{3, 101\}\)      & 0.0640 & 0 \\
            \(\{3, 103\}\)      & 0.0493 & 0 \\
            \(\{3, 107\}\)      & 0.0279 & 0 \\
            \(\{3, 109\}\)      & 0.1035 & 0 \\
            \(\{3, 113\}\)      & 0.0972 & 0 \\
            \(\{3, 127\}\)      & 0.0557 & 0 \\
            \(\{3, 131\}\)      & 0.0374 & 0 \\
            \(\{2, 3, 5\}\)     & 2.5539 & 6 \\
            \(\{2, 3, 7\}\)     & 2.5107 & 6 \\
            \(\{2, 3, 11\}\)    & 2.4462 & 6 \\
            \(\{2, 3, 13\}\)    & 2.3373 & 6 \\
            \(\{2, 3, 17\}\)    & 2.4355 & 6 \\
            \(\{2, 3, 19\}\)    & 2.3464 & 6 \\
            \(\{2, 3, 23\}\)    & 3.0774 & 6 \\
            \(\{3, 5, 7\}\)     & 0.9582 & 0 \\
            \(\{3, 5, 11\}\)    & 0.9134 & 0 \\
            \(\{3, 5, 13\}\)    & 1.3156 & 0 \\
            \(\{3, 5, 17\}\)    & 1.0217 & 0 \\
            \(\{3, 5, 19\}\)    & 1.0349 & 0 \\
            \(\{3, 5, 23\}\)    & 1.0629 & 0 \\
            \(\{3, 7, 11\}\)    & 1.0287 & 0 \\
            \(\{3, 7, 13\}\)    & 0.8835 & 0 \\
            \(\{3, 7, 17\}\)    & 1.1557 & 0 \\
            \(\{3, 7, 19\}\)    & 0.9422 & 0 \\
            \bottomrule
        \end{tabular}
    \end{minipage}
\end{table}

\printbibliography
\end{document}